\documentclass[11pt,reqno]{amsart}
\usepackage{amsthm,amssymb,amsmath}
\usepackage{textpos}

\usepackage[T1]{fontenc}

\newtheorem{theorem}{Theorem}
\newtheorem{lemma}{Lemma}

\newtheorem{proposition}{Proposition}
\newtheorem{claim}{Claim}

\theoremstyle{definition}

\newtheorem{definition}{\sc Definition}

\newtheorem{example}{\bf Example}
\newtheorem{remark}{\sc Remark}
\newtheorem*{example*}{\bf Example}

\newcommand{\loc}{{\rm loc}}
\newcommand{\supp}{{\rm spt\,}}

\newcommand{\Id}{{\rm Id}}

\newcommand{\clos}{{\rm clos}}

\expandafter\def\expandafter\normalsize\expandafter{%
    \normalsize
    \setlength\abovedisplayshortskip{8pt}
    \setlength\belowdisplayshortskip{8pt}
}

\begin{document}

\title[Feller evolution families]{
Feller evolution families and parabolic equations with form-bounded vector fields
}
\author{Damir Kinzebulatov}

\address{Universit\'{e} Laval, D\'{e}partement de math\'{e}matiques et de statistique, pavillon Alexandre-Vachon, Qu\'{e}bec, PQ, G1V 0A6, Canada}

\address{Current address: Indiana University, Department of Mathematics, Rawles Hall, Bloomington, IN, 47405, United States}

\email{damkinze@indiana.edu}



\begin{abstract}
We show that the weak solutions of parabolic equation 
$\partial_t u - \Delta u + b(t,x) \cdot \nabla u=0$, $(t,x) \in (0,\infty) \times \mathbb R^d$, $d \geqslant 3$,
for $b(t,x)$ in a wide class of time-dependent vector fields capturing critical order singularities, constitute a Feller evolution family and, thus, determine a Feller process. Our proof uses an a priori estimate on the $L^p$-norm of the gradient of solution in terms of the $L^q$-norm of the gradient of initial function, and an iterative procedure that moves the problem of convergence in $L^\infty$ to $L^p$.
\end{abstract}

\maketitle

\section{Introduction and results} 

\subsection{}
\label{introsect}

\label{mainsect}

Consider Cauchy problem 
\begin{equation}
\label{cauchy}
(\partial_t-\Delta + b(t,x)\cdot \nabla ) u=0, \qquad (t,x) \in (0,\infty) \times \mathbb R^d, 
\end{equation}
\begin{equation}
\label{weak_cauchy2}
u(+0,x)=f(x), 
\end{equation}
where $d \geqslant 3$, 
$b \in L^1_{\loc}([0,\infty) \times \mathbb R^d,\mathbb R^d)$, $f \in L^{2}_{\loc}(\mathbb R^d)$.

\smallskip

We prove that for $b$ in a wide class of time-dependent vector fields capturing critical order singularities the
unique weak solution of \eqref{cauchy}, \eqref{weak_cauchy2} for the
initial function $f$ in space
$C_\infty(\mathbb R^d):=\{f \in C(\mathbb R^d): \lim_{x \rightarrow \infty} f(x)=0\}$ (endowed with $\sup$-norm $\|\cdot\|_\infty$) is given
by a Feller  evolution family, i.e.~a family of bounded linear operators $(U(t,s))_{0 \leqslant s \leqslant t<\infty} \subset  \mathcal L\bigl(C_\infty(\mathbb R^d)\bigr)$ such that:

\begin{itemize}
\item[(\textbf{E1})] $U(s,s)=\Id$, $U(t,s)=U(t,r)U(r,s)$ for all $0 \leqslant s \leqslant r \leqslant t$,

\item[(\textbf{E2})] mapping $(t,s) \mapsto U(t,s)$ is strongly continuous in $C_\infty(\mathbb R^d)$,

\item[(\textbf{E3})] operators $U(t,s)$ are positivity-preserving and $L^\infty$-contractive:
\begin{equation*}
U(t,s)f \geqslant 0 \quad \text{ if } \quad f \geqslant 0, \quad \text{ and } \quad \|U(t,s)f\|_{\infty} \leqslant \|f\|_{\infty}, \quad 0 \leqslant s \leqslant t,
\end{equation*}
\end{itemize}

\begin{itemize}
\item[(\textbf{E4})] function $u(t):=U(t,s)f$ ($t>s$) is a weak solution of equation \eqref{cauchy}.

\end{itemize}

It is well known that the operators $(U(t,s))_{0 \leqslant s \leqslant t<\infty}$ determine the (sub-Markov) transition probability function of a Feller process $X_t$ (in particular, a Hunt process), see e.g.~\cite[Theorem 2.22]{GC}.
$X_t$ is related to the differential operator in \eqref{cauchy} via  (\textbf{E4}).
The problem of constructing a Brownian motion perturbed by a locally unbounded drift $b$
has been thoroughly studied in the literature, motivated by applications as well as by the search for the maximal general class of drifts $b$ such that the associated diffusion exists (see \cite{KR} and references therein). 

In the present paper, we consider the following class of drifts:

\begin{definition}
\label{fbd_def}
The parabolic class of form-bounded  vector fields $\mathbf{F}_{\beta,\,\mathcal P}=\mathbf{F}_{\beta,\,\mathcal P}(-\Delta)$
consists of vector fields $b \in L^2_{\loc}\bigl([0,\infty) \times \mathbb R^d,\mathbb R^d\bigr)$ such that
\begin{equation}
\label{bc}
\tag{$\mathbf{BC}$}
\int_0^\infty \|b(t,\cdot)\varphi(t,\cdot)\|_2^2 dt \leqslant \beta \int_0^\infty\|\nabla \varphi(t,\cdot)\|_2^2 dt+\int_0^\infty g(t)\|\varphi(t,\cdot)\|_2^2dt
\end{equation}
for some $\beta<\infty$ and $g=g_\beta \in L^1_{\loc}([0,\infty))$, $g \geqslant 0$,
for all $\varphi \in C_c^\infty([0,\infty)  \times \mathbb R^d)$.

$\|\cdot\|_2$ is the norm in $L^2(\mathbb R^d)$. 
\end{definition}

It is clear that $b \in \mathbf{F}_{\beta,\,\mathcal P} \,\Leftrightarrow\, cb \in \mathbf{F}_{c^2 \beta,\,\mathcal P}$, $c \neq 0$.

\begin{example}
\label{ex1}
1. If $b:\mathbb R^d \rightarrow \mathbb R^d$, $b=b_1+b_2$, $|b_1| \in L^{d,\infty}(\mathbb R^d)$ (weak $L^d$ space), $|b_2| \in L^\infty(\mathbb R^d)$, then $b \in \mathbf{F}_{\beta,\,\mathcal P}$ with 
\begin{align*}
\sqrt{\beta}=\|b_1\|_{d,\infty} \Omega_d^{-\frac{1}{d}} \frac{2}{d-2}, \qquad \Omega_d:=\pi^{\frac{d}{2}}\Gamma\left(\frac{d}{2}+1\right)
\end{align*}
(using Strichartz inequality with sharp constants \cite[Prop 2.5, 2.6, Cor.~2.9]{KPS}). 
In particular, $b(x)=x|x|^{-2}$ belongs to $\mathbf{F}_{\beta,\,\mathcal P}$ with $\beta=\left(2/(d-2) \right)^2$ (and $g \equiv 0$) (Hardy inequality). 
More generally, any vector field $b(t,x)$ such that for some $c_1$, $c_2>0$
$$
|b(t,x)|^2 \leqslant c_1|x-x_0|^{-2} + c_2|t-t_0|^{-1}\bigl(\log(e+|t-t_0|^{-1}) \bigr)^{-1-\varepsilon}, \quad \varepsilon>0, \quad (t,x) \in [0,\infty) \times \mathbb R^d,
$$
belongs to the class $\mathbf{F}_{\beta,\,\mathcal P}$ with
$\beta=c_1\left(2/(d-2) \right)^2$.
The above examples show that
the Gaussian bounds on the fundamental solution of $\partial_t-\Delta + b(t,x)\cdot\nabla$,  $b \in \mathbf{F}_{\beta,\,\mathcal P}$, are, in general, not valid. 

\smallskip

2.~If $h \in L^2(\mathbb R)$, $T:\mathbb R^d \rightarrow \mathbb R$ is a linear map, then the vector field $b(x)=h(Tx)a$, where $a \in \mathbb R^d$, is in $\mathbf{F}_{\beta,\,\mathcal P}$ with appropriate $\beta$, but $|b|$ may not be in $L^{d,\infty}_{\loc}(\mathbb R^d)$.

\smallskip
 
3.~Let $b:\mathbb R^d \rightarrow \mathbb R^d$. If $b^2$ is in the Campanato-Morrey class
$$
M_p:=\left\{v \in L^p: \|v\|_{M_p}:=\sup_{x \in \mathbb R^d, r>0} r^{2-\frac{d}{p}}\|\mathbf{1}_{B(x,r)}v\|_p < \infty\right\}
$$
for some $p>1$, then $b \in \mathbf{F}_{\beta,\,\mathcal P}$ with $\beta=\beta(\|b^2\|_{M_p})$. Here $\mathbf{1}_{B(x,r)}$ is the characteristic function of the open ball of radius $r$ centered at $x$.
 
\smallskip 
 
4. Set $L^qL^p:=L^q\bigl([0,\infty),L^p(\mathbb R^d)+L^\infty(\mathbb R^d)\bigr)$. 
We have: $$|b| \in L^qL^p \text{ with } \frac{d}{p}+\frac{2}{q} \leqslant 1 \qquad \Rightarrow \qquad b \in \mathbf{F}_{0,\,\mathcal P}:=\bigcap_{\beta>0}\mathbf{F}_{\beta,\,\mathcal P}$$
(using the H\"{o}lder inequality and the Sobolev embedding theorem).

\end{example}

The class $\mathbf{F}_{\beta,\,\mathcal P}$ contains vector fields having critical order singularities: replacing a $b \in \mathbf{F}_{\beta,\,\mathcal P}$ in \eqref{cauchy} with $cb$, $c>1$,
in general destroys e.g.~the uniqueness of weak solution of  Cauchy problem \eqref{cauchy}, \eqref{weak_cauchy2} (see \cite[Example 5]{KS}).
The class $\mathbf{F}_{0,\,\mathcal P}$ doesn't contain vector fields having critical order singularities.

The explicit dependence on the value of the relative bound $\beta$ is a crucial feature of our results.

\smallskip

We consider only real Banach spaces. Throughout this paper we use the following notation:
$$
\bigl\langle g \bigr\rangle=\bigl\langle g(\cdot) \bigr\rangle:=\int_{\mathbb R^d}g(x)dx.
$$
Let $\langle g, h\rangle$ denote the $(L^p,L^{p'})$ pairing, so that
$$\bigl\langle g\,,\,h\bigr\rangle:=\int_{\mathbb R^d}g(x)h(x)dx \qquad \bigl(g \in L^p(\mathbb R^d), h \in L^{p'}(\mathbb R^d)\bigr).$$

Before
formulating the main result, let us remind the reader the definition of a weak solution to Cauchy problem \eqref{cauchy}, \eqref{weak_cauchy2}.

\begin{definition}
A real-valued function $u \in L^\infty_{\loc}((0,\infty), L_{\loc}^2(\mathbb R^d))$ is said to be a weak solution of  equation \eqref{cauchy} if $\nabla u$ (understood in the sense of distributions) is in $L_{\loc}^1((0,\infty)\times \mathbb R^d, \mathbb R^d)$, $b \cdot \nabla u \in L_{\loc}^1((0,\infty)\times \mathbb R^d)$, and
\begin{equation}
\label{int_id}
\int_{0}^\infty \langle u, \partial_t \psi \rangle dt - \int_{0}^\infty \langle u,\Delta \psi \rangle dt +\int_{0}^\infty \langle b\cdot \nabla u ,\psi \rangle dt=0
\end{equation}
for all $\psi \in C_c^\infty((0,\infty) \times \mathbb R)$.
\end{definition}

\begin{definition}
A weak solution of \eqref{cauchy} is said to be a weak solution to Cauchy problem \eqref{cauchy}, \eqref{weak_cauchy2}
if
$
\lim_{t \rightarrow +0} \langle u(t), \xi \rangle=\langle f,\xi\rangle
$
for all $\xi \in L^{2}(\mathbb R^d)$ having compact support.
\end{definition}

\begin{theorem}[Main result]
\label{mainthm}
Let $d \geqslant 3$. Suppose a vector field
 $b(\cdot,\cdot)$ belongs to the class $\mathbf{F}_{\beta,\,\mathcal P}$.
If 
$\beta<d^{-2}$,
then there exists a Feller evolution family $(U(t,s))_{0 \leqslant s \leqslant t} \subset \mathcal L\bigl(C_\infty(\mathbb R^d) \bigr)$ that produces the weak solution to Cauchy problem \eqref{cauchy}, \eqref{weak_cauchy2}, i.e.~{\rm(\textbf{E1})--(\textbf{E4})} hold true.

\end{theorem}

Theorem \ref{mainthm} in the stationary case $b:\mathbb R^d \rightarrow \mathbb R^d$
and 
under the extra assumption $|b| \in L^2(\mathbb R^d) + L^\infty(\mathbb R^d)$
is due to \cite{KS}. 
The extra assumption is used there in the verification that
the constructed limit of approximating semigroups is strongly continuous in $C_\infty(\mathbb R^d)$ (i.e.~in the verification of the assumptions of the Trotter approximation theorem in $C_\infty(\mathbb R^d)$). We run their iterative procedure differently, so that it automatically yields strong continuity. 
(Generally speaking, unless $b$ is sufficiently regular in $t$, the non-stationary case presents the next level of difficulty compared to the stationary case.
It is the inherent flexibility of the method of \cite{KS} 
(which, we believe, goes beyond $\partial_t-\Delta + b(t,x)\cdot\nabla$) 
that allows us to carry out the construction of the process for a non-stationary $b(\cdot,\cdot) \in \mathbf{F}_{\beta,\,\mathcal P}$.)

Let us also note that, in the assumptions of Theorem \ref{mainthm}, given $p>(1-\sqrt{\beta/4})^{-1}$, the formula
\begin{equation*}
U_p(t,s):=\biggl(U(t,s)|_{L^p(\mathbb R^d) \cap C_\infty(\mathbb R^d)}\biggr)^{\clos}_{L^p(\mathbb R^d) \rightarrow L^p(\mathbb R^d)},
\end{equation*}
determines a (strongly continuous) evolution family in 
$\mathcal L\bigl(L^p(\mathbb R^d)\bigr)$, cf.~\cite{Se}.
The proof is obtained from Theorem \ref{mainthm}, estimate \eqref{est_Lp} below
and the Dominated Convergence Theorem.

\smallskip

We now briefly comment on the relationship between this work and the existing results.

1.~First, for
$|b| \in L^qL^p$ (cf.~Example \ref{ex1}.3), $ \frac{d}{p}+\frac{2}{q}< 1,$
the associated diffusion has been constructed in \cite{KR} as the strong solution of the SDE $dX_t=b(t,X_t)dt+\frac{1}{2}dW_t$, $X_0=x_0 \in \mathbb R^d$.

2.
Recall the definition of the parabolic Kato class $\mathbf{K}_{\beta,\,\mathcal P}^{d+1}$:
$$
\mathbf{K}_{\beta,\,\mathcal P}^{d+1}:=\left\{b \in L^1_{\loc}([0,\infty) \times \mathbb R^d,\mathbb R^d): \inf_{r>0} k^{1,1}(b,r) \leqslant \beta,\,\, \inf_{r>0} k^\infty(b,r) \leqslant \beta\right\},
$$
where
$$
k^{1,1}(b,r):=\sup_{u \geqslant 0,\,x \in \mathbb R^d} \int_{u}^{u+r} \int_{\mathbb R^d} \Gamma_{t-u}(x-y)\frac{|b(t,y)|}{\sqrt{t-u}} dy dt,
$$
$$
k^{\infty}(b,r):=
\sup_{u \geqslant r,\,x \in \mathbb R^d} \int_{u}^{u+r} \int_{\mathbb R^d} \Gamma_{u+r-t}(x-y)\frac{|b(t-r,y)|}{\sqrt{u+r-t}} dy dt,
$$
and $\Gamma_t(z):=(4\pi t)^{-\frac{d}{2}}e^{-\frac{|z|^2}{4t}}$.
If $b \in \mathbf{K}_{\beta,\,\mathcal P}^{d+1}$ with $\beta>0$ sufficiently small, then the fundamental solution of \eqref{cauchy} admits local in time Gaussian upper and lower bounds, see \cite{Zh}, which, in turn, yield the corresponding Feller evolution family (in $C_b(\mathbb R^d):=\{f \in C(\mathbb R^d): \sup_{x}|f(x)|<\infty\}$ endowed with the $\sup$-norm).
Note that 
$\mathbf{K}_{0,\,\mathcal P}^{d+1} - \mathbf{F}_{\beta,\,\mathcal P} \neq \varnothing$, where $\mathbf{K}_{0,\,\mathcal P}^{d+1}:=\cap_{\beta>0}\mathbf{K}_{\beta,\,\mathcal P}^{d+1}$ (on the other hand, $L^d(\mathbb R^d,\mathbb R^d) - \mathbf{K}_{\beta,\,\mathcal P}^{d+1} \cap \{\mathsf f:\mathbb R^d \rightarrow \mathbb R^d\} \neq \varnothing$).

\smallskip

3.~In the stationary case $b:\mathbb R^d \rightarrow \mathbb R^d$, it has been shown in \cite{Ki} that the associated Feller process exists for vector fields $b$ in the class
$$
\mathbf{F}_\beta^{\frac{1}{2}}:=\left\{b \in L^1_{\loc}(\mathbb R^d,\mathbb R^d): \bigl\||b|^{\frac{1}{2}}(\lambda-\Delta)^{-\frac{1}{4}}\bigr\|^2_{L_2 \rightarrow L_2} \leqslant \sqrt{\beta}\,\,  \text{ for some }\lambda=\lambda_\beta>0\right\}.
$$
In particular, the class $\mathbf{F}_\beta^{\frac{1}{2}}$ contains vector fields of the form $b:=b_1+b_2$, where $b_1 \in \mathbf{F}_{\beta}:=\mathbf{F}_{\beta,\mathcal P} \cap \{\mathsf f:\mathbb R^d \rightarrow \mathbb R^d\}$,
$b_2 \in \mathbf{K}_\beta^{d+1}:=\mathbf{K}^{d+1}_{\beta,\,\mathcal P} \cap \{\mathsf f:\mathbb R^d \rightarrow \mathbb R^d\}$.

\begin{remark}
We leave out the $L^p$-theory of $\partial_t - \Delta + b(t,x)\cdot\nabla$ with $b \in \mathbf{F}_{\beta,\,\mathcal P}$, $1<\beta<4$, or with $b$ in a parabolic analogue of the class $\mathbf{F}_\beta^{\frac{1}{2}}$.
\end{remark}

\noindent\textbf{Acknowledgements.~}\textit{I am deeply grateful to Yu.A.~Semenov for many important comments, and constant attention throughout this work. }

\section{Proof of Theorem \ref{mainthm}}

\label{mainthmproof_sect}

We will need a regular approximation of $b$:  vector fields
$\{b_m\}_{m=1}^\infty \subset C_c^\infty([0,\infty) \times \mathbb R^d, \mathbb R^d)$ that satisfy
$b_m \rightarrow b$ in $L^2_{\loc}([0,\infty) \times \mathbb R^d, \mathbb R^d),$ and
\begin{equation}
\label{bc_m}
\tag{$\mathbf{BC}_m$}
\int_0^\infty \|b_{m}(t,\cdot)\varphi(t,\cdot)\|_2^2 dt \leqslant \left(\beta + \frac{1}{m} \right)\int_0^\infty\|\nabla \varphi(t,\cdot)\|_2^2 dt+\int_0^\infty g(t)\|\varphi(t,\cdot)\|_2^2dt
\end{equation}
for all $\varphi \in C_c^\infty([0,\infty) \times \mathbb R^d)$.
(Such $b_m$'s can be constructed by the formula
$
b_m:=\eta_m \ast \mathbf 1_m b, 
$
where $\mathbf 1_m$ is the characteristic function of set $\{(t,x) \in \mathbb R \times \mathbb R^d : |b(t,x)| \leqslant m, |x| \leqslant m, 0 \leqslant |t| \leqslant m\}$, $\ast$ is the convolution on $\mathbb R \times \mathbb R^d$, and $\{\eta_m\} \subset C_c^\infty(\mathbb R \times \mathbb R^d)$ is an appropriate family of mollifiers.) 

Due to the strict inequality $\beta<d^{-2}$, we may assume without loss of generality that $b_m$'s satisfy ($\mathbf{BC}_m$) with $\beta$ in place of $\beta + \frac{1}{m}$.

\smallskip

The construction of the Feller evolution family goes as follows. 
Fix some $T>0$. Denote $$D_T:=\{(s,t) \in \mathbb R^2: 0 \leqslant s \leqslant t \leqslant T\}.$$
Let $(U_m(t,s))_{0 \leqslant s \leqslant t} \subset \mathcal L(C_\infty(\mathbb R^d))$ be the Feller
evolution family for the equation
\begin{equation}
\label{approxeq}
(\partial_t -\Delta + b_m(t,x)\cdot \nabla)u=0.
\end{equation}
Given a $f \in C_c^\infty(\mathbb R^d)$, we define
\begin{equation}
\label{limU}
U f:=\lim_{m \rightarrow \infty} U_m f \quad \text{in} \quad L^\infty\bigl(D_T, C_\infty(\mathbb R^d)\bigr) 
\end{equation}
Assuming that the convergence in \eqref{limU} has been established, we note that $U_m$ is $L^\infty$-contractive and $C_c^\infty(\mathbb R^d)$ is dense in $C_\infty(\mathbb R^d)$, so $U=(U(t,s))_{0 \leqslant s \leqslant t}$ extends  to a strongly continuous family of bounded linear operators in $\mathcal L\bigl(C_\infty(\mathbb R^d)\bigr)$, which we denote again by $(U(t,s))_{0 \leqslant s \leqslant t}$.

\begin{proposition}
\label{lem0}
In the assumptions of Theorem \ref{mainthm} $(U(t,s))_{0 \leqslant s \leqslant t}$ defined by \eqref{limU} satisfies {\rm{\rm(}\textbf{E1}{\rm)}-{\rm(}\textbf{E4}{\rm)}}.
\end{proposition}

The main difficulty is in establishing the convergence in \eqref{limU}. The proof of the convergence uses a parabolic variant of the iterative procedure of \cite{KS}.

\subsection{Proof of the convergence in \eqref{limU}: a parabolic variant of the iterative procedure of Kovalenko-Semenov}

\label{itersect}

Fix $f \in C_c^\infty(\mathbb R^d)$.
Set $$u_m(t)=U_m(t,s)f, \quad t \geqslant s.$$

\begin{lemma} [a priori estimate]
\label{thmapr}
Let $d \geqslant 3$. Suppose $b$ is in $\mathbf{F}_{\beta,\,\mathcal P}$ with $\beta<d^{-2}$, $q \in \left(d,\beta^{-\frac{1}{2}}\right)$. Then
$$
\|\nabla u_m\|_{L^\infty([s,\tau],L^q(\mathbb R^d))} + C_1 \|\nabla u_m\|_{L^q([s,\tau],L^{\frac{qd}{d-2}}(\mathbb R^d))}
\leqslant C\|\nabla f\|_{q}, \quad s \leqslant \tau \leqslant T,
$$
where constants $C_1=C_1(q,\beta)>0$, $C=C(q,T)<\infty$, do not depend on $m$ or $(s,\tau)$.
\end{lemma}

\begin{remark}
The a priori estimate of Lemma \ref{thmapr} is one of the main results of the paper. It is the basis for the approach as a whole (for the corresponding result in the elliptic case see \cite[Lemma 5]{KS}).

\end{remark}

We subtract the approximating equations \eqref{approxeq} for $b_m$, $b_n$, and integrate to obtain:

\begin{lemma}
\label{lem2}
Suppose $b \in \mathbf{F}_{\beta,\,\mathcal P}$ with $\beta<4$. Let $0 < \alpha < 1$.
There exist $h>0$, $k=k(\beta)>1$ and a $m_0$ such that for all $m,n \geqslant m_0$,
for all $p \geqslant p_0>\frac{2}{2-\sqrt{\beta}}$ we have
\begin{multline}
\label{iterineq}
\|u_m-u_n\|_{L^{\frac{p}{1-\alpha}}([s,s+h],L^{\frac{pd}{d-2+2\alpha}}(\mathbb R^d))}\\ \leqslant \left(C_0 \beta \|\nabla u_m\|^2_{L^{2\lambda'}([s,s+h],L^{2\sigma'}(\mathbb R^d))}\right)^{\frac{1}{p}}(p^{2k})^{\frac{1}{p}}\|u_m-u_n\|_{L^{(p-2)\lambda}([s,s+h],L^{(p-2)\sigma}(\mathbb R^d))}^{1-\frac{2}{p}},
\end{multline}
for any $\sigma$ such that $1<\sigma<\frac{d}{d-2+2\alpha}$, $\frac{1}{\sigma}+\frac{1}{\sigma'}=1$, and $\frac{1/(1-\alpha)}{\lambda}=\frac{d/(d-2+2\alpha)}{\sigma}$, $\frac{1}{\lambda}+\frac{1}{\lambda'}=1,$ for a constant $C_0=C_0(h)<\infty$ that doesn't depend on $m$ or $s \leqslant T$.
\end{lemma}

The a priori estimate of Lemma \ref{thmapr} allows to iterate the inequality \eqref{iterineq} (with a proper choice of $\alpha$, $\lambda$ and $\sigma$) in order to obtain an $L^\infty$-norm in the left-hand side, and an $L^p$-norm ($p<\infty$) (of $u_m-u_n$) in the right-hand side. 
Set $$D_{T,\,h}:=D_{T} \cap \{(s,t):0 \leqslant t-s \leqslant h\}, \quad h<T.$$

\begin{lemma}
\label{lem3}
In the assumptions of Theorem \ref{mainthm}, 
for any $ p_0>\frac{2}{2-\sqrt{\beta}}$ there exist $h>0$, constants $B<\infty$ and $\gamma:=
\bigl( 1-\frac{\sigma d }{d+2}\bigr)\bigl(1-\frac{\sigma d }{d+2}+\frac{2\sigma}{p_0} \bigr)^{-1}
>0$ {\rm(}$1<\sigma<\frac{d+2}{d}${\rm)} independent of $m,n$ such that
\begin{equation}
\label{iterineq0}
\|U_mf-U_nf\|_{L^\infty(D_{T,\,h} \times \mathbb R^d)}  \leqslant B \sup_{0 \leqslant s \leqslant T-h}\|U_mf-U_nf\|^\gamma_{L^{p_0}([s,s+h], L^{p_0}(\mathbb R^d))} \quad \text{ for all } n,m.
\end{equation}
\end{lemma}

\begin{remark}
Lemma \ref{lem3} is the key result. It moves the problem of convergence of $\{U_mf\}$ in $L^\infty$ to a space having much weaker topology (locally).
\end{remark}

That $\{U_mf\}$ does indeed converge in the weaker topology of the right-hand side of \eqref{iterineq0}
will follow from the following

\begin{lemma}
\label{lem1}
Suppose $b \in \mathbf{F}_{\beta,\,\mathcal P}$ with  $\beta<1$.
The sequence $\{U_mf\}$ from Lemma \ref{lem3} is fundamental in $L^\infty(D_T, L^r(\mathbb R^d))$, $2 \leqslant r <\infty$.
\end{lemma}

\medskip

Let us prove the convergence in \eqref{limU}. Fix $f \in C_c^\infty(\mathbb R^d)$,
and choose  $r=2$ in Lemma \ref{lem1}. Then $r>\frac{2}{2-\sqrt{\beta}}$ since $\beta$ is less than $1$, and we can
take $p_0:=r$ in Lemma \ref{lem3}. Now,
Lemma \ref{lem3} and Lemma \ref{lem1} imply that there exists $h>0$ such that the sequence $\{U_mf\}$ is fundamental in $L^\infty(D_{T,\,h}, C_\infty(\mathbb R^d))$. By the reproduction property, $\{U_mf\}$ is fundamental in 
$L^\infty(D_T, C_\infty(\mathbb R^d))$.
The convergence in \eqref{limU} follows.

 The proof of Theorem \ref{mainthm} is completed. 

\begin{remark}
Note that the constraint on $\beta$ in Theorem \ref{mainthm}  (in addition to $\beta<1$)  comes solely from Lemma \ref{thmapr}.
\end{remark}

\section{Proofs of Lemmas \ref{thmapr} -- \ref{lem1} and Proposition \ref{lem0}}
\label{proof_sect}

\subsection*{Preliminaries}

\textbf{1.~}We will use the following well known fact (which we use below for $u_m$). 
Suppose that $b$ belongs to $\mathbf{F}_{\beta,\,\mathcal P}$ with $\beta<1$. If
$p>(1-\sqrt{\beta/4})^{-1}$, $f \in L^p(\mathbb R^d)$,
then the (unique) weak solution $u$ of the equation \eqref{cauchy} such that
$$
\lim_{t \rightarrow +0} \langle u(t), \xi \rangle=\langle f,\xi\rangle
$$
for all $\xi \in L^{p'}(\mathbb R^d)$ having compact support, $\frac{1}{p}+\frac{1}{p'}=1$, satisfies
\begin{equation}
\label{est_Lp}
\sup_{t \in [0,\tau]}  \|u(t)\|_p^p + C_1\int_0^\tau \langle (\nabla (u |u|^{\frac{p}{2}-1}))^2\rangle dt 
\leqslant C_2\|f\|^p_p,
\end{equation}
where $0<C_i=C_i(\beta,g,p)<\infty$, $i=1,2$ (see Appendix A for the proof for $u_m$ which, in turn, is sufficient to conclude \eqref{est_Lp} for $u$ as above).

\smallskip

\textbf{2.~}Let $g$ be the function from the condition \eqref{bc}. Set 
$$
G(h):=\sup_{0 \leqslant s \leqslant T-h}\int_s^{s+h} g(t)dt.
$$
Clearly, $G(h)=o(h)$ (i.e.~$G(h) \rightarrow 0$ as $h\rightarrow 0$).

\smallskip

\subsection*{Proof of Lemma \ref{thmapr}}

It suffices to prove Lemma \ref{thmapr} for $s \leqslant \tau \leqslant s + h$, for a small $h$, uniformly in $s$.

We consider smooth approximating vector fields $b_m:=\eta_m \ast \mathbf{1}_m b$, not just truncations $\mathbf{1}_mb$ of $b$ (cf.~the beginning of Section \ref{mainthmproof_sect}), because the intermediate  calculations below involve third order derivatives of $u$.

In what follows, we omit index $m$ where possible:
$
u(t):=u_m(t)~\left(=U_m(t,s)f, t \geqslant s \right).
$
Denote $w=\nabla u$, $w_r=\frac{\partial}{\partial x_r}u$, $1 \leqslant r \leqslant d$.
Define
$$
\varphi_r:=-\frac{\partial}{\partial x_r}\left(w_r|w|^{q-2} \right), \quad 1 \leqslant r \leqslant d,
$$
$$
I_q = \int_s^\tau \left\langle |w|^{q-2} \sum_{r=1}^d \left|\nabla w_r\right|^2\right\rangle dt \geqslant 0, \quad J_q = \int_s^\tau \langle |w|^{q-2}|\nabla |w||^2 \rangle dt \geqslant 0.
$$
Now, we are going to `differentiate the equation without differentiating its coefficients'. That is, we multiply the  equation in (\ref{cauchy}) by the `test function' $\varphi_r$, integrate in $t$ and $x$, and then sum over $r$ to get
\begin{equation*}
S:=\sum_{r=1}^d \int_s^\tau \left\langle \varphi_r, \frac{\partial u}{\partial t} \right \rangle dt  = \sum_{r=1}^d \int_s^\tau \langle \varphi_r, \Delta u \rangle dt - \sum_{r=1}^d\int_s^\tau \langle \varphi_r, b_m \cdot w\rangle dt =:S_1+S_2.
\end{equation*}
We can re-write
$$
S=\frac{1}{q}\int_s^\tau \frac{\partial}{\partial t} \left\langle |w|^{q}\right\rangle dt = \frac{1}{q} \langle |w(\tau)|^{q} \rangle - \frac{1}{q} \langle |\nabla f|^{q} \rangle
$$
(the fact that $w(s)=\nabla f$ follows by differentiating in $x_i$, for each $1 \leq i \leq d$, the equation in (\ref{cauchy}) and the initial function $f$, solving the resulting Cauchy problem, and then integrating its solution in $x_i$ to see that it is indeed the derivative of $v$ in $x_i$). 
Further, 
\begin{align*}
S_1&=-\sum_{r=1}^d \int_s^\tau \left\langle\frac{\partial}{\partial x_r}\left(w_r|w|^{q-2}\right),\Delta u \right\rangle dt = -\sum_{r=1}^d \int_s^\tau \left\langle \nabla \left(w_r|w|^{q-2} \right), \nabla w_r\right\rangle dt \\
&=-\int_s^\tau \left\langle |w|^{q-2} \sum_{r=1}^d \left|\nabla w_r\right|^2\right\rangle dt - \frac{1}{2}\int_s^\tau \langle \nabla |w|^{q-2},\nabla |w|^2 \rangle dt =- I_q - (q-2)J_q.
\end{align*}
Next, 
\begin{equation*}
S_2 = \int_s^\tau \langle |w|^{q-2} \Delta u, b_m \cdot w \rangle dt +
 \int_s^\tau \left \langle w \cdot \nabla |w|^{q-2}, b_m \cdot w\right\rangle dt =: W_1+W_2.
\end{equation*}
Let us estimate $W_1$ and $W_2$ as follows. 
By the inequality $ac \leq \frac{\gamma}{4} a^2 + \frac{1}{\gamma}c^2$ ($\gamma>0$), we have
\begin{align*}
& |W_1| \leqslant \int_s^\tau \langle |w|^{\frac{q-2}{2}}|\Delta u| |w|^{\frac{q-2}{2}}|b_m||w| \rangle dt  \\
& \leqslant \frac{\gamma}{4} \int_s^\tau \langle |w|^{q-2}|\Delta u|^2 \rangle dt + 
\frac{1}{\gamma} \int_s^\tau \left\langle \left( |b_m| |w|^{\frac{q}{2}} \right)^2\right\rangle dt  \\
& \text{ (we use  ($\mathbf{BC}_m$), where we omit $1/m$ in $\beta+1/m$)} \\
& \leqslant \frac{\gamma}{4} \int_s^\tau \langle |w|^{q-2}|\Delta u|^2 \rangle dt + \frac{1}{\gamma} \left [\beta \frac{q^2}{4}J_q + \int_s^\tau g(t) \langle |w|^{q} \rangle \right] 
\end{align*}
In turn, representing $|\Delta u|^2=(\nabla \cdot w)^2$
and integrating by parts twice we obtain:
\begin{multline*}
\int_s^\tau \langle |w|^{q-2}|\Delta u|^2 \rangle dt  = -\int_s^\tau \langle  w \cdot \nabla |w|^{q-2}, \Delta u \rangle dt
+\sum_{r=1}^d \int_s^\tau \left\langle w \cdot \nabla w_r, \nabla_r |w|^{q-2}\right\rangle dt 
+ I_q\\ =:-F+H+I_q,
\end{multline*}
where we estimate, using quadratic estimates of the form $ac \leq \varkappa a^2 + \frac{1}{4\varkappa}c^2$ ($\varkappa>0$),
$$
|F| \leqslant (q-2)\left( \frac{1}{4\varkappa}  \int_s^\tau \langle |w|^{q-2}|\Delta u|^2\rangle dt + \varkappa J_q\right),
\quad |H| \leqslant (q-2)\left( \frac{1}{2} I_q+ \frac{1}{2} J_q\right).
$$

Thus, we obtain
\begin{equation*}
\biggl(1-\frac{q-2}{4\varkappa} \biggr)\int_s^\tau \langle |w|^{q-2}|\Delta u|^2 \rangle dt \leqslant I_q+(q-2)\left(\varkappa J_q + \frac{1}{2} I_q + \frac{1}{2} J_q \right), \quad \varkappa>\frac{q-2}{4},
\end{equation*}
so
$$
|W_1| \leqslant \frac{\gamma}{4}\frac{4\varkappa}{4\varkappa-q+2}\biggl(I_q+(q-2)\left(\varkappa J_q + \frac{1}{2} I_q + \frac{1}{2} J_q \right) \biggr) + \frac{1}{\gamma} \left [\beta \frac{q^2}{4} J_q + \int_s^\tau g(t) \langle |w|^{q} \rangle \right].
$$

Next, using $ac \leq \nu a^2 + \frac{1}{4\nu}c^2$ ($\nu>0$), we obtain
\begin{align*}
& |W_2| \leqslant (q-2) \int_s^\tau \langle |w|^{q-2} |\nabla |w|| |b_m||w| \rangle dt = (q-2) \int_s^\tau \langle |w|^{\frac{q-2}{2}} |\nabla |w|| |b_m||w|^{\frac{q}{2}} \rangle dt \\
& \leqslant 
(q-2) \left[ \nu \int_s^\tau \langle |w|^{q-2} |\nabla |w||^2\rangle dt + 
\frac{1}{4\nu} \int_s^\tau \left\langle \left(|b_m||w|^{\frac{q}{2}}\right)^2 \right\rangle  dt \right] \\ 
  & \text{ (we use  ($\mathbf{BC}_m$))}\\
& \leqslant (q-2) \left[ \nu J_q
+ 
\frac{\beta}{4\nu} \frac{q^2}{4}J_q +
\frac{1}{4\nu} \int_s^\tau g(t)\langle |w|^{q}\rangle dt \right].
\end{align*}
Thus, identity $S=S_1+S_2$ transforms into
\begin{equation*}
\frac{1}{q} \langle |w(\tau)|^q\rangle - \frac{1}{q} \langle |\nabla f|^q\rangle + I_q + (q-2) J_q = W_1+W_2,
\end{equation*}
and, in view of the above estimates on $|W_1|$, $|W_2|$,  implies
\begin{equation}
\label{ooo}
\frac{1}{q} \langle |w(\tau)|^{q}\rangle + N\,I_q + 
M\,J_q \leqslant \\
\frac{1}{q} \langle |\nabla f|^{q} \rangle + 
\left(\frac{q-2}{4\nu}+\frac{1}{\gamma}\right)\int_s^\tau g(t) \langle |w|^{q} \rangle dt,
\end{equation}
where
$$
N:=1 - \frac{\gamma\varkappa}{4\varkappa-q+2}\bigl(1+\frac{1}{2}(q-2)\bigr),
$$
\begin{equation*}
M := \\ q-2  - (q-2)\left(\nu + \frac{\beta}{16 \nu}q^2 \right) - \frac{\beta}{\gamma}  \frac{q^2}{4} - \frac{\gamma\varkappa}{4\varkappa-q+2}(q-2)\left(\varkappa+\frac{1}{2}\right).
\end{equation*}
We fix
$$
\nu:=q\sqrt{\beta}/4, \quad \varkappa := \frac{q-1}{2}, \quad \gamma := \frac{q\sqrt{\beta}}{q-1}.
$$
Since $\sqrt{\beta}<q^{-1}$, we have $N>0$. Then, in view of the
 inequality $I_q \geqslant J_q$, we have
$$
N\,I_q + M\,J_q \geqslant \biggl(q-1-\frac{q\sqrt{\beta}}{2}(2q-3)\biggr)J_q, \quad \text{ where, clearly, }q-1-\frac{q\sqrt{\beta}}{2}(2q-3)>\frac{1}{2}.
$$

Then, applying the Sobolev embedding theorem  to $\frac{q^2}{4}J_q~(=\int_s^\tau \langle |\nabla|w|^{\frac{q}{2}}|^2 \rangle dt$), and recalling that $w=\nabla u$, we obtain from \eqref{ooo}:
\begin{equation*}
\frac{1}{q} \langle |\nabla u(\tau)|^{q}\rangle + \frac{2C_d}{q^2}\|\nabla u\|^q_{L^q([s,\tau],L^{\frac{qd}{d-2}}(\mathbb R^d))} \\ \leqslant 
\frac{1}{q} \langle |\nabla f|^{q} \rangle +
\left(\frac{q-2}{4\nu}+\frac{1}{\gamma}\right)\int_s^\tau g(t) \langle |\nabla u(t)|^{q} \rangle dt,
\end{equation*}
where $C_d>0$ is the constant in the Sobolev embedding theorem. 

Estimating $\int_s^\tau g(t) \langle |\nabla u|^{q} \rangle dt \leqslant G(h)\sup_{t \in [s,\tau]} \langle |\nabla u(t)|^{q}\rangle $, and  selecting $h\, (\geqslant \tau-s)$ sufficiently small, so that $\left(\frac{q-2}{4\nu}+\frac{1}{\gamma}\right)G(h)<\frac{1}{2q}$ (recall that $G(h)=o(h)$, cf.~the beginning of Section \ref{proof_sect}), we obtain
\begin{equation*}
\frac{1}{2} \sup_{t \in [s,\tau]}\langle |\nabla u(t)|^{q}\rangle  + \frac{2C_d}{q}\|\nabla u\|^q_{L^q([s,\tau],L^{\frac{qd}{d-2}}(\mathbb R^d))}  \leqslant 
 \langle |\nabla f|^{q} \rangle,
\end{equation*}
which completes the proof.

\subsection*{Proof of Lemma \ref{lem2}}

Set $r=r_{m,n}:=u_m-u_n$.
Then $r$ satisfies
\begin{equation}
\label{lem2_id}
\partial_t r=\Delta r - b_m(t,x)\cdot\nabla r - \bigl(b_m(t,x)-b_n(t,x)\bigr)\cdot\nabla u_n.
\end{equation}
Set
$\eta:=r|r|^{\frac{p-2}{2}}$. 
We multiply equation \eqref{lem2_id} by $r|r|^{p-2}$ and integrate to obtain the identity 
\begin{multline}
\label{lem2_id2}
\frac{1}{p}\|\eta(\tau)\|_2^2+\frac{4(p-1)}{p^2}\int_{s}^\tau \|\nabla \eta\|_2^2 dt= 
-\frac{2}{p}\int_s^\tau \langle \nabla \eta, b_m \eta \rangle dt -
\int_{s}^\tau \langle \eta|\eta|^{1-\frac{2}{p}},(b_m-b_n)\cdot\nabla u_n\rangle dt
\end{multline}
(note that by definition $\eta(s) \equiv 0$). We estimate the right-hand side of \eqref{lem2_id2}. Using $ac \leqslant \varepsilon a^2+\frac{1}{4\varepsilon }c^2$ ($\varepsilon>0$) and ($\mathbf{BC}_m$), we obtain:
\begin{align*}
\left|\int_{s}^\tau \langle \nabla \eta,b_m\eta \rangle dt\right| & \leqslant 
\varepsilon \int_{s}^\tau \langle (b_m\eta)^2\rangle dt+\frac{1}{4\varepsilon}\int_{s}^\tau \langle |\nabla \eta|^2 \rangle dt \\ 
&\leqslant 
\varepsilon \beta\int_{s}^\tau \langle |\nabla \eta|^2 \rangle dt+\varepsilon\int_{s}^\tau g(t) \langle \eta^2 \rangle dt+\frac{1}{4\varepsilon} \int_{s}^\tau \langle |\nabla \eta|^2 \rangle dt.
\end{align*}
Next, using  $|b_m-b_n| \leqslant |b_m|+|b_n|$, $ac \leqslant \delta a^2+\frac{1}{4\delta }c^2$ ($\delta>0$), and  ($\mathbf{BC}_m$), we find
\begin{align*}
\left| \int_{s}^\tau \langle \eta|\eta|^{1-\frac{2}{p}}, (b_m-b_n)\cdot\nabla u_n \rangle dt \right| & \leqslant \int_{s}^\tau \langle |b_m-b_n||\eta|,|\eta|^{1-\frac{2}{p}}|\nabla u_n|\rangle dt \\
&\leqslant \delta\int_{s}^\tau \langle (b_m\eta)^2 \rangle dt +\delta\int_{s}^\tau \langle (b_n\eta)^2\rangle dt +2\frac{1}{4\delta}\int_{s}^\tau \langle |\eta|^{2-\frac{4}{p}} |\nabla u_n|^2\rangle dt \\ 
& \leqslant 2\delta\left( \beta\int_{s}^\tau \langle |\nabla \eta|^2\rangle dt + \int_{s}^\tau g(t) \langle \eta^2 \rangle dt \right) + 2\frac{1}{4\delta}\int_{s}^\tau \langle |\eta|^{2-\frac{4}{p}}|\nabla u_n|^2\rangle dt.
\end{align*}
Thus, applying the last two estimates in the right-hand side of \eqref{lem2_id2}, we obtain:
\begin{multline*}
\frac{1}{p} \|\eta(\tau)\|_2^2  + \left(\frac{4(p-1)}{p^2}-\frac{2}{p}\left(\varepsilon\beta+\frac{1}{4\varepsilon}\right)-2\beta\delta \right)\int_{s}^\tau \langle |\nabla \eta|^2 \rangle dt \\
\leqslant \frac{1}{2\delta}\int_{s}^\tau \langle |\eta|^{2-\frac{4}{p}}|\nabla u_n|^2\rangle dt + \left(\frac{2}{p}\varepsilon+2\delta \right)\int_{s}^\tau g(t)\langle \eta^2 \rangle dt.
\end{multline*}
Set $$P:=\frac{4(p-1)}{p^2}-\frac{2}{p}\left(\varepsilon\beta+\frac{1}{4\varepsilon}\right)-2\beta\delta \qquad \text{ with } \varepsilon:=\frac{1}{2\sqrt{\beta}}.$$

Estimating $\int_{s}^\tau g(t)\langle \eta^2 \rangle dt \leqslant G(h)\sup_{t \in [s,\tau]}\|\eta(t)\|_{2}^2$, we have:
\begin{equation}
\label{doublestar}
\left(\frac{1}{p}-\biggl(\frac{1}{p\sqrt{\beta}}+2\delta\biggr) G(h) \right) \sup_{t \in [s,\tau]}\|\eta(t)\|_2^2  + P \int_{s}^\tau \langle |\nabla \eta|^2 \rangle dt 
\leqslant \frac{1}{2\delta}\int_{s}^\tau \langle |\eta|^{2-\frac{4}{p}}|\nabla u_n|^2\rangle dt.
\end{equation}
Since $p_0 > \frac{2}{2-\sqrt{\beta}}$, we can fix $k$ so that
$\frac{4(p_0-1)}{p_0^2}-\frac{2}{p_0}\sqrt{\beta} \geqslant \frac{2}{p_0^k}$.
The last inequality remains valid if we replace $p_0$ with any $p> p_0$.
Fix $\delta$ by
$$\delta:=\frac{1}{2\beta}\left(\frac{4(p-1)}{p^2}-\frac{2}{p}\sqrt{\beta} - \frac{1}{p^k}\right) \geqslant \frac{1}{2\beta p^k}.
$$
Then
\begin{equation*}
P = \frac{4(p-1)}{p^2} - \frac{2}{p}\sqrt{\beta} -2\beta \delta 
=\frac{1}{p^k}.
\end{equation*}

In the next Steps 1 and 2 we estimate the left-hand side and the right-hand side of \eqref{doublestar}.

\medskip

Step 1.~Given $0 < \alpha < 1$, we can choose $k>1$ so that for all $n \geqslant m_0$,
\begin{equation}
\label{iii}
\frac{c_0}{p^k}\|r\|_{L^{\frac{p}{1-\alpha}}([s,\tau],L^{\frac{pd}{d-2+2\alpha}}(\mathbb R^d))}^p \leqslant 
\text{the LHS of \eqref{doublestar}}.
\end{equation}
for some constant $c_0<\infty$.

Indeed, applying the Sobolev embedding theorem in the spatial variables, 
we obtain from \eqref{doublestar}:
$$
\left(\frac{1}{p}-\left(\frac{1}{p\sqrt{\beta}}+2\delta \right)G(h) \right)\sup_{t \in [s,\tau]}\|r(t)\|_{p}^p + \frac{C_d}{p^k}\|r\|_{L^p([s,\tau],L^{\frac{pd}{d-2}}(\mathbb R^d))}^{p} \leqslant \text{the LHS of \eqref{doublestar}}.
$$
Since $\delta \leqslant \frac{c}{p}$, $c:=\frac{1}{\beta}(2-\sqrt{\beta})$, we can select $h$ sufficiently small (we use that $G(h)=o(h)$), so that for all $p \geqslant p_0$
\begin{align*}
\frac{1}{p}-\left(\frac{1}{p\sqrt{\beta}}+2\delta \right)G(h) & \geqslant  \\
\frac{1}{p}\left(1-\left(\frac{1}{\sqrt{\beta}}+2c\right)G(h)\right) & \geqslant  \frac{1}{2p}  \\
& \text{(we use that $k>1$)}  \\
& \geqslant  \frac{1}{2p^k}.
\end{align*}
Thus, we have
$$
\frac{1}{2 p^k}\sup_{t \in [s,\tau]}\|r(t)\|_{p}^p + \frac{C_d}{p^k}\|r\|_{L^p([s,\tau],L^{\frac{pd}{d-2}}(\mathbb R^d))}^{p} \leqslant \text{the LHS of \eqref{doublestar}}.
$$
Using first the H\"{o}lder inequality, and then the Young inequality we obtain:
\begin{align*}
\|r\|_{L^{\frac{p}{1-\alpha}}([s,\tau],L^{\frac{pd}{d-2+2\alpha}}(\mathbb R^d))}^p &  \leqslant \|r\|_{L^{\infty}([s,\tau],L^p(\mathbb R^d))}^{\alpha p}   \|r\|_{L^p([s,\tau],L^{\frac{pd}{d-2}}(\mathbb R^d))}^{(1-\alpha)p} \\
& \leqslant \alpha \|r\|_{L^{\infty}([s,\tau],L^p(\mathbb R^d))}^{p} + (1-\alpha) \|r\|_{L^p([s,\tau],L^{\frac{pd}{d-2}}(\mathbb R^d))}^{p},
\end{align*}
which yields \eqref{iii}.

\medskip

Step 2:~With $\sigma$,
$\sigma'$ 
and $\lambda$, $\lambda'$ 
as in the formulation of the lemma,
we have
\begin{equation}
\label{iii_}
\text{the RHS of \eqref{doublestar}} \leqslant 
\beta p^k \|\nabla u_n\|^2_{L^{2\lambda'}([s,\tau],L^{2 \sigma'}(\mathbb R^d))} \|r\|^{p-2}_{L^{(p-2)\lambda}([s,\tau],L^{(p-2)\sigma}(\mathbb R^d))}
\end{equation}
Indeed, since $\delta \geqslant \frac{1}{2\beta p^k}$, the RHS of \eqref{doublestar}
$
=\frac{1}{2\delta}\int_{s}^\tau \langle |\eta|^{2-\frac{4}{p}}|\nabla u_n|^2\rangle dt \leqslant \beta p^k \int_{s}^\tau \langle |\eta|^{2-\frac{4}{p}}|\nabla u_n|^2\rangle dt.
$
In turn,
\begin{align*}
&\int_{s}^\tau \langle |\eta|^{2-\frac{4}{p}}|\nabla u_n|^2\rangle dt \leqslant \int_s^\tau  \langle |\nabla u_n|^{2\sigma'}\rangle^{\frac{1}{\sigma'}} \langle |\eta|^{\left( 2-\frac{4}{p}\right)\sigma}\rangle^{\frac{1}{\sigma}} dt \\
&=\int_s^\tau  \|\nabla u_n\|^2_{L^{2\sigma'}(\mathbb R^d)} \|r\|^{p-2}_{L^{(p-2)\sigma}(\mathbb R^d)}dt \\
&\leqslant \left( \int_s^\tau  \|\nabla u_n\|^{2\lambda'}_{L^{2\sigma'}(\mathbb R^d)} dt \right)^{\frac{1}{\lambda'}} \left(\int_s^\tau \|r\|^{(p-2)\lambda}_{L^{(p-2)\sigma}(\mathbb R^d)} dt \right)^{\frac{1}{\lambda}} \\  
&= \|\nabla u_n\|^2_{L^{2\lambda'}([s,\tau],L^{2 \sigma'}(\mathbb R^d))} \|r\|^{p-2}_{L^{(p-2)\lambda}([s,\tau],L^{(p-2)\sigma}(\mathbb R^d))},
\end{align*}
which yields \eqref{iii_}.

\medskip

Substituting the estimates \eqref{iii} and \eqref{iii_} into \eqref{doublestar}, and taking $\tau:=s+h$, we  arrive at the required estimate \eqref{iterineq}.

\subsection*{Proof of Lemma \ref{lem3}}

The proof of Lemma \ref{lem3} follows closely the proof of \cite[Lemma 7]{KS}. 
Consider the inequality of Lemma \ref{lem2}:
\begin{multline}
\label{iterineq_}
\|u_m-u_n\|_{L^{\frac{p}{1-\alpha}}([s,s+h],L^{\frac{pd}{d-2+2\alpha}}(\mathbb R^d))} \\ \leqslant \left(C_0 \beta \|\nabla u_m\|^2_{L^{2\lambda'}([s,s+h],L^{2\sigma'}(\mathbb R^d))}\right)^{\frac{1}{p}}(p^{2k})^{\frac{1}{p}}\|u_m-u_n\|_{L^{(p-2)\lambda}([s,s+h],L^{(p-2)\sigma}(\mathbb R^d))}^{1-\frac{2}{p}},
\end{multline}
where $\lambda$ is defined by
$\frac{1/(1-\alpha)}{\lambda}=\frac{d/(d-2+2\alpha)}{\sigma}$, and $\frac{1}{\lambda}+\frac{1}{\lambda'}=1$ (it is easy to see that $\lambda'=\frac{\sigma'(d-2+2\alpha)}{d(1-\alpha)}$).
We fix $\alpha:=\frac{2}{d+2}$
(we keep $\alpha$ to make the calculations easier to follow) and $1<\sigma<\frac{d}{d-2+2\alpha}$ so that $\sigma'>\frac{d}{2(1-\alpha)}$, determined from $\frac{1}{\sigma}+\frac{1}{\sigma'}=1$, is sufficiently close to $\frac{d}{2(1-\alpha)}$. 
\smallskip
We apply the a priori estimate of Lemma \ref{thmapr}:
\begin{align*}
&\|\nabla u_m\|^2_{L^{2\lambda'}([s,s+h],L^{2\sigma'}(\mathbb R^d))} \\
& \text{(we use the H\"{o}lder inequality)} \\
& \leqslant \|\nabla u_m\|^{\alpha}_{L^\infty([s,s+h],L^q(\mathbb R^d))} \|\nabla u_m\|^{1-\alpha}_{L^q([s,s+h],L^{\frac{qd}{d-2}}(\mathbb R^d))} \\
& \text{(we use Young's inequality)}\\
& \leqslant \alpha \|\nabla u_m\|_{L^\infty([s,s+h],L^q(\mathbb R^d))} + (1-\alpha) \|\nabla u_m\|_{L^q([s,s+h],L^{\frac{qd}{d-2}}(\mathbb R^d))} \\
& \text{(we use Lemma \ref{thmapr})} \\ 
& \leqslant C\|\nabla f\|_q=:D<\infty,
\end{align*}
where $q$ is determined from $\sigma'=\frac{1}{2}\frac{qd}{d-2+2\alpha}$ (such $q$ ($\in (d,1/\sqrt{\beta})$) in Lemma \ref{thmapr} is admissible, in view of the assumptions on $\beta$ in Theorem \ref{mainthm}).
Then \eqref{iterineq_} yields
\begin{equation}
\label{iterineq2}
\|u_m-u_n\|_{L^{\frac{p}{1-\alpha}}([s,s+h],L^{\frac{pd}{d-2+2\alpha}}(\mathbb R^d))} \leqslant D^{\frac{1}{p}}(p^{2k})^{\frac{1}{p}}\|u_m-u_n\|_{L^{(p-2)\lambda}([s,s+h],L^{(p-2)\sigma}(\mathbb R^d))}^{1-\frac{2}{p}}.
\end{equation}

In order to iterate the inequality \eqref{iterineq2},
choose any $p_0>\frac{2}{2-\sqrt{\beta}}$ and construct a sequence $\{p_l\}_{l \geqslant 0}$ by successively assuming $\sigma(p_1-2)=p_0$,
$\sigma(p_2-2)=\frac{p_1d}{d-2+2\alpha}$, $\sigma(p_3-2)=\frac{p_2d}{d-2+2\alpha}$ etc, so that 
\begin{equation}
\label{p_l}
p_l=(a-1)^{-1}\left(a^l\left(\frac{p_0}{\sigma}+2\right)-a^{l-1}\frac{p_0}{\sigma}-2\right), \quad a:=\frac{1}{\sigma}\frac{d}{d-2+2\alpha}>1.
\end{equation}
Clearly, 
\begin{equation}
\label{p_l2}
c_1 a^l \leqslant p_l \leqslant c_2 a^l, \quad \text{ where } c_1:=p_1 a^{-1}, \quad c_2:=c_1 (a-1)^{-1},
\end{equation}
and so 
$p_l \rightarrow \infty$ as $l \rightarrow \infty$.

Now, we iterate inequality \eqref{iterineq2}, starting with $p=p_0$, to obtain
\begin{equation}
\label{iterineq4}
\|u_m-u_n\|_{L^{\frac{p_l}{1-\alpha}}([s,s+h],L^{\frac{p_ld}{d-2+2\alpha}}(\mathbb R^d))} \leqslant D^{\alpha_l }\Gamma_l\|u_m-u_n\|_{L^{p_0\lambda}([s,s+h],L^{p_0\sigma}(\mathbb R^d))}^{\gamma_l}, 
\end{equation}
where $$\gamma_l:=\left(1-\frac{2}{p_1} \right) \dots \left(1-\frac{2}{p_l} \right),$$
\begin{multline*}
\alpha_l:=\frac{1}{p_1}\biggl(1-\frac{2}{p_2} \biggr)\biggl(1-\frac{2}{p_3} \biggr) \dots \biggl(1-\frac{2}{p_l} \biggr) + \\ 
\frac{1}{p_2}\biggl(1-\frac{2}{p_3} \biggr)\biggl(1-\frac{2}{p_4} \biggr) \dots \biggl(1-\frac{2}{p_l} \biggr) + \dots +
\frac{1}{p_{l-1}}\biggl(1-\frac{2}{p_l} \biggr) + \frac{1}{p_{l}},
\end{multline*}
$$
\Gamma_l:=\biggl(p_l^{p_l^{-1}}p_{l-1}^{p_{l-1}^{-1} (1-2p_l^{-1})} p_{l-2}^{p_{l-2}^{-1}(1-2p_{l-1}^{-1})(1-2p_l^{-1})} \dots p_1^{p_1^{-1}(1-2p_2^{-1}) \dots (1-2p_l^{-1})} \biggr)^{2k}.
$$
We wish to take $l \rightarrow \infty$ in \eqref{iterineq4}: since $p_l \rightarrow \infty$ as $l \rightarrow \infty$, this would yield the required inequality \eqref{iterineq0} provided  that sequences $\{\alpha_l\}$, $\{\Gamma_l\}$ are bounded from above, and $\{\gamma_l\}$ is bounded from below by a positive constant.
Note that
$
\alpha_l=a^l-\frac{1}{p_l(a-1)}$, $\gamma_l=p_0 \frac{a^{l-1}}{\sigma p_l}.
$
In view of \eqref{p_l}, 
\begin{equation}
\label{need1}
\sup_l \alpha_l \leqslant \biggl(\frac{p_0}{\sigma}+2-\frac{p_0(d-2+2\alpha)}{d}\biggr)^{-1}<\infty, \quad \sup_l \gamma_l<\infty, 
\end{equation}
\begin{equation}
\label{need2}
\inf_l \gamma_l> \biggl( 1-\frac{\sigma(d-2+2\alpha)}{d}\biggr)\biggl(1-\frac{\sigma(d-2+2\alpha)}{d}+\frac{2\sigma}{p_0} \biggr)^{-1}>0.
\end{equation}
Further, noticing that (cf.~\eqref{p_l}) $\Gamma_l^{1/2k}=p_l^{p_l^{-1}}p_{l-1}^{a p_l^{-1}} p_{l-2}^{a^2 p_l^{-1}} \dots p_1^{a^{l-1}p_l^{-1}}$, we have by \eqref{p_l2}
\begin{multline}
\label{need3}
\Gamma_l^{1/2k} \leqslant (c_1 a^l)^{(c_2a^l)^{-1}} (c_1 a^{l-1})^{(c_2a^{l-1})^{-1}} \dots (c_1 a)^{(c_2 a)^{-1}}= \\ \biggl(c_1^{(a^l-1)/(a^l(a-1))} a^{\sum_{j=1}^l j a^{-j}} \biggr)^{c_2^{-1}} \leqslant 
\biggl(c_1^{(a-1)^{-1}} c_2^{a(a-1)^{-1}} \biggr)^{c_2^{-1}}<\infty.
\end{multline}
Now, estimates \eqref{need1}, \eqref{need2} and \eqref{need3} imply that we can take $l \rightarrow \infty$ in \eqref{iterineq2}:
\begin{equation*}
\|u_m-u_n\|_{L^\infty([s,s+h],L^{\infty}(\mathbb R^d))} \leqslant B\|u_m-u_n\|_{L^{p_0}([s,s+h],L^{p_0}(\mathbb R^d))}^{\gamma}.
\end{equation*}
Taking $\sup$ in $0 \leqslant s \leqslant T-h$ in both sides of the inequality, we obtain \eqref{iterineq0} in Lemma \ref{lem3}.

\begin{remark}
The main concern of the iterative procedure has been to keep $\inf_l \gamma_l>0$: if $\gamma_l \downarrow 0$, then the result of the iterations ($\|U_mf-U_nf\|_{L^\infty(D_T \times \mathbb R^d)} \leqslant C$) would be useless for the purpose of proving Theorem \ref{mainthm}.
\end{remark}

\subsection*{Proof of Lemma \ref{lem1}}

By the reproduction property, 
and in view of \eqref{est_Lp},
it suffices to show that  $\{U_mf\}$ is fundamental in $L^\infty(D_{T,\,h}, L^2(\mathbb R^d))$ for some $h>0$. We show this in three steps:

\smallskip

\textbf{Step 1.} 
Define $$\rho_{\delta}(x):=(1+\delta |x|^2)^{-\frac{1}{2}}, \quad \delta>0,  \quad x \in \mathbb R^d.$$

In Step 1, we are going to show that there is an $h=h(g)>0$ ($g$ is from the condition ($\mathbf{BC}_m$)) such that for any $\varepsilon>0$ there is a $0<\delta<1$ such that
\begin{equation}
\label{reqineq1}
\|(1-\rho_\delta)^{\frac{1}{2}}\,U_mf\|_{L^\infty(D_{T,\,h},L^2(\mathbb R^d))}<\varepsilon \quad \text{ for all }m.
\end{equation}

Indeed, set $u_m(t)=U_m(t,s)f$ ($t \geqslant s$). Set
$$
J:=\int_s^\tau \langle (1-\rho_\delta) (\nabla u_m)^2 \rangle dt.
$$
We multiply the equation in \eqref{cauchy} by $(1-\rho_\delta)u_m$ and integrate by parts to get
\begin{equation}
\label{main_id}
\langle (1-\rho_\delta) u_m^{2}(\tau)\rangle - \langle (1-\rho_\delta)f^{2}\rangle + 2 J
= \int_s^\tau \langle u_m^2, (-\Delta \rho_\delta) \rangle dt - 2\int_s^\tau \langle (1-\rho_\delta)u_{m} b_m, \nabla u_m \rangle dt.
\end{equation}
Estimating the last term by applying the inequality $2ac \leqslant \gamma a^2 + \frac{1}{\gamma}c^2$ ($\gamma>0$) and the condition ($\mathbf{BC}_m$), we get:
\begin{align*}
& - 2\int_s^\tau \langle (1-\rho_\delta)u_{m} b_m, \nabla u_m \rangle dt  \\
& \leqslant
\gamma J
+\frac{1}{\gamma}
\int_s^\tau \langle(1-\rho_\delta) b_m^2 u_m^2  \rangle dt \\
& \leqslant 
\gamma J
+ \frac{\beta}{\gamma} \int_s^\tau \langle (\nabla (u_m \sqrt{1-\rho_\delta}))^2 \rangle dt
+ \frac{1}{\gamma} \int_s^\tau \langle g(t)(1-\rho_\delta)  u_m^{2} \rangle dt.
\end{align*}
We compute:
\begin{align*}
& \int_s^\tau \langle (\nabla (u_m \sqrt{1-\rho_\delta}))^2 \rangle dt \\
& = J + \int_s^\tau \langle u^2 (\nabla \sqrt{1-\rho_\delta})^2 \rangle dt + \frac{1}{2} \int_s^\tau \langle u^2, (-\Delta \rho_\delta) \rangle dt \\ 
& = J + \int_s^\tau \left\langle u^2, \frac{\delta^2 x^2 \rho^6}{4(1-\rho)}\right\rangle dt + \int_s^\tau \left\langle u^2, \frac{\rho^3\delta}{2}(d - 3\rho^2 \delta x^2)\right\rangle dt.
\end{align*}
Thus, estimating $\int_s^\tau \langle g(t)(1-\rho_\delta)  u_m^{2} \rangle dt \leqslant G(h) \sup_{t \in [s,\tau]}\langle (1-\rho_\delta) u_m^{2}(t)\rangle$, we obtain from \eqref{main_id}:
\begin{align*}
&\left(1-\frac{G(h)}{\gamma} \right)\sup_{t \in [s,\tau]}\langle (1-\rho_\delta) u_m^{2}(t)\rangle + \left(2-\gamma-\frac{\beta}{\gamma}  \right)J \\
&\leqslant \langle (1-\rho_\delta)f^{2}\rangle + \frac{\beta}{\gamma}\int_s^\tau \left\langle u^2, \frac{\delta^2 x^2 \rho^6}{4(1-\rho)}\right\rangle dt + \left(1-\frac{\beta}{\gamma}\right) \int_s^\tau \left\langle u^2, \frac{\rho^3\delta}{2}(d - 3\rho^2 \delta x^2)\right\rangle dt.
\end{align*}
Now, fix $\gamma>0$ by the condition $2-\gamma-\frac{\beta}{\gamma}>0$,
and then fix $h$ by the condition $1-\frac{1}{\gamma}G(h)>0$ (recall that $G(h)=o(h)$). Noting that 
$\frac{\delta^2 x^2 \rho^6(x)}{4(1-\rho(x))} \leqslant \frac{\delta}{2} \rho(x)$, $\frac{\rho^3(x)\delta}{2}\bigl(d - 3\rho^2(x) \delta x^2\bigr) \leqslant \delta \frac{d-3}{2}\rho(x)$,  $\int_s^\tau \langle \rho_\delta u^2 \rangle dt  \leqslant h C \|f\|_2^2$ (by \eqref{est_Lp} with $p=2$), we obtain:
\begin{equation*}
\left(1-\frac{G(h)}{\gamma} \right)\sup_{t \in [s,\tau]}\langle (1-\rho_\delta) u_m^{2}(t)\rangle  
\leqslant \langle (1-\rho_\delta)f^{2}\rangle + \delta h C \left( \frac{\beta}{2\gamma} + \left(1-\frac{\beta}{\gamma}\right)\frac{d-3}{2}\right)\|f\|_2^2.
\end{equation*}
Since $\rho_{\delta} \rightarrow 1$ uniformly on the support of $f \in C_c^\infty(\mathbb R^d)$ as $\delta \rightarrow 0$, the right-hand side of the inequality can be made
arbitrarily small by taking  sufficiently small $\delta$, i.e.~we have proved \eqref{reqineq1}.

\medskip

\textbf{Step 2.}~
In Step 2, we are going to show that there is an $h=h(g)>0$ such that for a given $\varepsilon>0$ and $\delta:=\delta(\varepsilon)$ from Step 1 
there is a $n_0=n_0(\varepsilon)$ such that
\begin{equation}
\label{omega_3}
\bigl\| \rho_\delta^{\frac{1}{2}} (U_mf - U_nf)\bigr\|_{L^\infty(D_{T,\,h},L^2(\mathbb R^d))} < \varepsilon \quad \text{ for all } 
m,n \geqslant n_0.
\end{equation}

Indeed, by the equation for $r(t):=u_m(t)-u_n(t)~(=U_m(t,s)f-U_n(t,s)f)$,
\begin{equation*}
\int_s^\tau \left\langle \rho_\delta r \frac{\partial r}{\partial t}\right\rangle dt + \int_s^\tau \langle \rho_\delta r (-\Delta r) \rangle dt =- \int_s^\tau \langle \rho_\delta r,  b_m \cdot  \nabla r\rangle dt - \int_s^\tau \langle \rho_\delta r, (b_m-b_n) \cdot \nabla u_n\rangle dt.
\end{equation*}
Integrating by parts in the second term in the left-hand side, and applying the inequality $ac \leqslant \frac{1}{2} a^2 +\frac{1}{2} c^2$ to the first term in the right-hand side,
we obtain:
\begin{equation*}
\langle \rho_\delta r^2(\tau)\rangle   + \int_s^\tau \langle \rho_\delta (\nabla r)^2 \rangle dt   + 2\int_s^\tau \langle r\nabla \rho_\delta, \nabla r\rangle dt \\ \leqslant \int_s^\tau \langle \rho_\delta b_m^2 r^2\rangle dt -
2\int_s^\tau \langle \rho_\delta r, (b_m-b_n) \cdot \nabla u_n\rangle dt
\end{equation*}
or 
$$
\langle \rho_\delta r^2(\tau)\rangle + \int_s^\tau \langle \rho_\delta (\nabla r)^2 \rangle dt + K \leqslant L + Z.
$$

We have
$$
K=\int_s^\tau \langle \nabla \rho_\delta, \nabla r^2 \rangle dt = \int_s^\tau \langle (-\Delta \rho_\delta) r^2 \rangle dt = \int_s^\tau \langle \left(\delta d \rho_\delta^3 - 3\delta^2 |x|^2 \rho_\delta^5  \right) r^2 \rangle  dt \geqslant 0.
$$

Next, using ($\mathbf{BC}_m$)
we obtain
\begin{align*}
& L = \int_s^\tau \langle \rho_\delta b_m^2 r^2\rangle dt \\
& \leqslant \beta\int_s^\tau \langle (\nabla(\sqrt{\rho_\delta} r))^2\rangle dt +\int_s^\tau g(t) \langle \rho_\delta r^2 \rangle dt \\ 
& \left(\text{here we use $\frac{(\nabla \rho_\delta(x))^2}{\rho_\delta(x)}=\delta^2 |x|^2 \rho^5$}\right) \\
&=
\frac{\beta}{4} \int_s^\tau  \langle \delta ^2 |x|^2 \rho_\delta^5 r^2 \rangle dt + \frac{\beta}{2} K + \beta \int_s^\tau \langle \rho_\delta (\nabla r)^2 \rangle dt + \int_s^\tau g(t) \langle \rho_\delta r^2 \rangle dt.
\end{align*}

Now we combine the above bound on $L$ and the estimates $$\int_s^\tau g(t) \langle \rho_\delta r^2 \rangle dt \leqslant G(h) \sup_{t \in [s,\tau]}\langle \rho_\delta r^2(t)\rangle, \quad \int_s^\tau  \langle \delta ^2 |x|^2 \rho_\delta^5 r^2 \rangle dt \leqslant h \,\delta \sup_{t \in [s,\tau]}\langle \rho_\delta r^2(t)\rangle,$$ obtaining:
\begin{equation}
\label{K}
\left(1-G(h)-\frac{\beta \delta h}{4} \right)\sup_{t \in [s,\tau]}\langle \rho_\delta r^2(t)\rangle + \left(1-\beta\right)\int_s^\tau \langle \rho_\delta (\nabla r)^2 \rangle dt + \left(1-\frac{\beta}{2}\right) K \leqslant  Z.
\end{equation}
Fix $h>0$ by the condition $1-G(h)-\frac{\beta \delta h}{4} \geqslant \frac{1}{2}$ (recall that $G(h)=o(h)$, $\beta$, $\delta<1$).

Finally, we estimate the term $Z$ as follows:
\begin{align*}
&Z =-2\int_s^\tau \langle \rho_\delta r (b_m-b_n), \nabla u_n\rangle dt  \\
& \leqslant  \varepsilon \int_s^\tau (\nabla u_n)^2 dt + \frac{1}{\varepsilon}\int_s^\tau \langle \rho_\delta^2 r^2 (b_m-b_n)^2\rangle dt \\
& \text{(here we use $\int_s^\tau (\nabla u_n)^2 dt \leqslant C\|f\|_2^2$, see Appendix A with $p=2$)} \\
& \leqslant  
\varepsilon C \|f\|^2_2 + \frac{1}{\varepsilon}\int_s^\tau \langle \rho_\delta^2 r^2 (b_m-b_n)^2\rangle dt, 
 \\
 & \leqslant \varepsilon C \|f\|^2_2 + \frac{1}{\varepsilon}\int_s^\tau \langle (1-\mathbf{1}_{B(0,R)}) \rho_\delta^2 r^2 (b_m-b_n)^2\rangle dt + \frac{1}{\varepsilon}\int_s^\tau \langle \mathbf{1}_{B(0,R)} \rho_\delta^2 r^2 (b_m-b_n)^2\rangle dt \\
& =:\varepsilon C \|f\|^2_2 + \frac{1}{\varepsilon} Z_{1} + \frac{1}{\varepsilon} Z_{2}.
\end{align*}
In turn,
$$
Z_{1} \leqslant 2(1+\delta R^2)^{-\frac{1}{2}} \left( \int_s^\tau \langle \rho_\delta b_m^2 r^2\rangle dt + \int_s^\tau \langle \rho_\delta b_n^2 r^2\rangle dt \right).
$$
Estimating the terms in the brackets in the last inequality in the same way as $L$, and substituting the resulting estimate on $Z$ into \eqref{K}, we obtain:
\begin{multline*}
\left(1-G(h)-\frac{\beta \delta h}{4}-\frac{1}{\varepsilon}(1+\delta R^2)^{-\frac{1}{2}}C_1\right)\sup_{t \in [s,\tau]}\langle \rho_\delta r^2(t)\rangle \\ 
+\left(1-\beta-\frac{4\beta}{\varepsilon}(1+\delta R^2)^{-\frac{1}{2}}\right)\int_s^\tau \langle \rho_\delta (\nabla r)^2 \rangle dt + \left(1-\frac{\beta}{2}-\frac{2\beta}{\varepsilon}(1+\delta R^2)^{-\frac{1}{2}}\right) K
\leqslant  \varepsilon C \|f\|^2_2 + \frac{1}{\varepsilon}Z_{2},
\end{multline*}
where $C_1:=4\left(G(h) + \frac{\beta \delta h}{4} \right)$.

Choose $R=R(\varepsilon,\delta)>0$ sufficiently large to ensure that the coefficients of $\int_s^\tau \langle \rho_\delta (\nabla r)^2 \rangle dt$, $K$ remain positive and, moreover, the coefficient of $\sup_{t \in [s,\tau]}\langle \rho_\delta r^2(t)\rangle$ is greater or equal to $\frac{1}{4}$ (since $1-G(h)-\frac{\beta \delta h}{4} \geqslant \frac{1}{2}$). Then the previous inequality yields
\begin{equation}
\label{K2}
\frac{1}{4}\sup_{t \in [s,\tau]}\langle \rho_\delta r^2(t)\rangle \leqslant \varepsilon C \|f\|^2_2 + \frac{1}{\varepsilon}Z_{2}.
\end{equation}
Since $U_m$ is $L^\infty$-contractive, $\|r(\tau)\|_\infty \leqslant 2\|f\|_{\infty}$ 
and so there is a $n_0=n_0(R,\varepsilon)$ such that
\begin{align*}
Z_{2}&=\int_s^\tau \langle \mathbf{1}_{B(0,R)} \rho_\delta^2 r^2 (b_m-b_n)^2\rangle dt \\
& \leqslant 4\|f\|^2_{\infty}  \int_s^\tau \langle \mathbf{1}_{B(0,R)} (b_m-b_n)^2\rangle dt < \varepsilon^2
\end{align*}
for all $(s,\tau) \in D_{T,\,h}$ for all $m,n \geqslant n_0$
since $b_m \rightarrow b$ in $L^2_{\loc}([s,s+h] \times \mathbb R^d,\mathbb R^d)$.

Thus, in view of \eqref{K2}
$$
\sup_{t \in [s,\tau]}\langle \rho_\delta r^2(t)\rangle < 4(C\|f\|^2_2+1)\,\varepsilon.
$$
Therefore, we have proved \eqref{omega_3}.

\smallskip

\textbf{Step 3.~}Set $\|\cdot\|:=\|\cdot\|_{L^\infty(D_{T,\,h},L^2(\mathbb R^d))}$. The results of Step 1 and Step 2 yield: for any $\varepsilon>0$ there is a $\delta=\delta(\varepsilon)<1$, and an $n_0=n_{0}(\varepsilon)$ such that
\begin{align*}
&\|U_mf-U_nf\|^2 
= \|(1-\rho_\delta)^{\frac{1}{2}}(U_mf-U_nf)\|^2 + \|\rho_\delta^{\frac{1}{2}}(U_mf-U_nf)\|^2 \\
& \leqslant 2\|(1-\rho_\delta)^{\frac{1}{2}}U_mf \|^2 +
2\|(1-\rho_\delta)^{\frac{1}{2}}U_nf \|^2  + \|\rho_\delta^{\frac{1}{2}}(U_mf-U_nf)\|^2 < 5\varepsilon
\end{align*}
for all $m,n \geqslant n_0$. 

The latter implies that
$\{U_mf\}$ is fundamental  in $L^\infty(D_{T,\,h},L^2(\mathbb R^d))$, as required.

\subsection*{Proof of Proposition \ref{lem0}}

In Section \ref{itersect} we proved the existence of $U f:={\scriptstyle L^\infty(D_T \times \mathbb R^d)}{\text{\rm-}}\lim_{m \rightarrow \infty} U_m f$, $f \in C_c^\infty(\mathbb R^d)$.
Since $C_c^\infty(\mathbb R^d)$ in dense in $C_\infty(\mathbb R^d)$, and $U_m$ is $L^\infty$-contractive,
$U$ extends by continuity to $C_\infty(\mathbb R^d)$.  Thus,
the property (\textbf{E2}) is established.

The properties (\textbf{E1}) and (\textbf{E3})
follow from \eqref{limU} and the analogous properties 
of $U_m$.

We are left to prove (\textbf{E4}). Set $u(t)=U(t,0)f$ ($t \geqslant 0$), $f \in C_\infty(\mathbb R^d)$. 
In order to verify that $u$ is a weak solution of \eqref{cauchy},
we have to show that $b \cdot \nabla u \in L_{\loc}^1((0,\infty)\times \mathbb R^d)$. Since $b \in L^{2}_{\loc}([0,\infty) \times \mathbb R^d, \mathbb R^d)$, it suffices to show that $\nabla u \in L^{2}_{\loc}((0,\infty) \times \mathbb R^d, \mathbb R^d)$.
Fix $k>\frac{d}{2}$. Set $$\theta_{\delta}(x):=(1+\delta |x|^2)^{-k}, \quad \delta >0,  \quad x \in \mathbb R^d.$$
It is easy to see that $\theta_{\delta} \in L^1(\mathbb R^d)$.

Set $u_m(t)=U_m(t,0)f$ ($t \geqslant 0$). 

\begin{claim}
\label{dontrun2}
There exist an $h>0$ and a $\delta>0$ such that for all $m$
\begin{equation}
\label{L2_3}
\int_0^h \langle \theta_\delta(\nabla u_m)^2 \rangle dt  \leqslant c_1 \langle \theta_\delta f^2  \rangle \\ + c_2 \sqrt{\delta}\|f\|^2_\infty, \quad f \in C_\infty(\mathbb R^d),
\end{equation}
where constants $c_1, c_2<\infty$ don't depend on $m$.

\end{claim}

\begin{proof}[Proof of Claim \ref{dontrun2}]
For all $m$,
\begin{equation}
\label{L2_4}
C_0\int_0^h \langle \theta_\delta(\nabla u_m)^2\rangle dt  \leqslant \langle \theta_\delta f^2 \rangle \\ + C_1 k \sqrt{\delta} \biggl( \int_0^h \langle \theta_\delta u_m^2  \rangle dt + \int_0^h \langle \theta_\delta(\nabla u_m)^2  \rangle dt \biggr),
\end{equation}
where $0<C_0, C_1<\infty$ do not depend on $m$ or $\delta$. 
The proof is similar to the proof of Lemma \ref{lem1}\,(Step 1) but with $1-\rho_\delta$ replaced by $\theta_\delta$. By \eqref{L2_4},
\begin{equation*}
(C_0-C_1k\sqrt{\delta})\int_0^h \langle \theta_\delta(\nabla u_m)^2\rangle dt  \leqslant \langle \theta_\delta f^2  \rangle \\ + C_1k\sqrt{\delta} \int_0^h \langle \theta_\delta u_m^2 \rangle dt \quad \text{ for all }m.
\end{equation*}
We choose $\delta>0$ by the condition $C_0-C_1k\sqrt{\delta}>0$.
Recalling that $U_m$ is $L^\infty$-contractive and $\theta_\delta \in L^1$, we obtain $\int_0^h \langle \theta_\delta u_m^2  \rangle dt \leqslant C_3  \|f\|^2_\infty$.
This yields \eqref{L2_3}.
\end{proof}

We fix $h$ and $\delta$ from Claim \ref{dontrun2}. By \eqref{L2_3},
the sequence $\{\nabla u_m|_{[0,h] \times \bar{B}(0,R)}\}$ is weakly relatively compact in $L^2([0,h] \times \bar{B}(0,R),\mathbb R^d)$, where $\bar{B}(0,R)$ is the closed ball of radius $R>0$ arbitrarily fixed. Hence, $\nabla u|_{(0,h) \times B(0,R)}$ (understood in the sense of distributions) is in $L^2([0,h] \times \bar{B}(0,R),\mathbb R^d)$. It follows that $\nabla u \in L^{2}_{\loc}((0,\infty) \times \mathbb R^d, \mathbb R^d)$. 

(Note that if $f \in C_\infty(\mathbb R^d) \cap L^2(\mathbb R^d)$, then $\nabla u \in L^{2}_{\loc}((0,\infty) \times \mathbb R^d, \mathbb R^d)$ also follows from \eqref{est_Lp} with $p=2$.)
 
\smallskip 

It remains to show that $u$ satisfies the integral identity \eqref{int_id}. Clearly, 
\begin{equation}
\label{appl_id}
\int_0^\infty \langle u_{m}, \partial_t \psi \rangle dt - \int_0^\infty \langle u_{m},\Delta \psi \rangle dt + \int_0^\infty \langle (b_m-b)\cdot \nabla u_{m} ,\psi \rangle dt + \int_0^\infty \langle b\cdot \nabla u_{m} ,\psi \rangle dt=0.
\end{equation}
Without loss of generality, we consider only the test functions $\psi$ with $\supp \,\psi \subset (0,h) \times B(0,R)$,  for some $R>0$. 
Since $u_{m} \rightarrow u$ in $C([0,h],C_\infty(\mathbb R^d))$ by \eqref{limU}, we can pass to the limit $m \rightarrow \infty$ in the first two terms in the left-hand side of \eqref{appl_id}. By the H\"{o}lder inequality,
\begin{equation*}
\biggl|\int_{0}^{\infty} \langle (b_m-b)\cdot \nabla u_m ,\psi \rangle dt  \biggr|
 \leqslant S^{\frac{1}{2}} \left(\int_{0}^{\infty} \langle (b_m-b)^2|\psi| \rangle dt \right)^{\frac{1}{2}},
\end{equation*}
where $S:=\sup_m \int_{s}^{T} \langle |\nabla u_m|^2|\psi| \rangle dt<\infty$ by \eqref{L2_3}.
Therefore, since $b_m \rightarrow b$ in $L^2_{\loc}([0,\infty) \times \mathbb R^d, \mathbb R^d)$ and $\supp\psi$ is compact, the third term the left-hand side of  \eqref{appl_id} tends to $0$ as $m \rightarrow \infty$.
Finally, we can pass to the limit $m \rightarrow \infty$ in the fourth term in  \eqref{appl_id} because $\{\nabla u_m|_{[0,h] \times \bar{B}(0,R)}\}$ is weakly relatively compact in $L^2([0,h] \times \bar{B}(0,R))$, see \eqref{L2_3}, and $|b\psi| \in L^2([0,h] \times \bar{B}(0,R))$.

\appendix

\section{}

\begin{proof}[Proof of \eqref{est_Lp}]
We omit index $m$: $u=u_m$.
Without loss of generality, we may assume that $\tau \leqslant h$ for a small $h$, and that $f \geqslant 0$, so $u \geqslant 0$. Multiply the equation \eqref{cauchy} by $u^{p-1}$ and integrate to get
\begin{equation*}
R:=\int_0^\tau \langle u^{p-1}, \partial_t u \rangle dt  = \int_0^\tau \langle u^{p-1},\Delta u\rangle dt - \int_0^\tau \langle u^{p-1},b_m \cdot \nabla u\rangle dt =:R_1+R_2.
\end{equation*}
We have
$$
R=\frac{1}{p}\langle u^p(\tau)\rangle - \frac{1}{p}\langle f^p\rangle, \quad R_1=-(p-1)\frac{4}{p^2} \int_0^\tau \langle (\nabla u^{\frac{p}{2}})^2 \rangle dt.
$$
Using the inequality $ac \leqslant \nu a^2 + \frac{1}{4\nu} c^2$ ($\nu>0$) and the condition ($\mathbf{BC}_m$), we obtain:
\begin{equation*}
R_2 = -\frac{2}{p} \int_0^\tau \langle u^{\frac{p}{2}}, b_m \cdot \nabla u^{\frac{p}{2}}\rangle dt
\leqslant \frac{2}{p}\nu \int_0^\tau \langle (\nabla u^{\frac{p}{2}})^2 \rangle dt +  \frac{1}{2p\nu}\biggl(\beta \int_0^\tau \langle (\nabla u^{\frac{p}{2}})^2 \rangle dt +  \int_0^\tau \langle g(t)u^p\rangle dt \biggr).
\end{equation*}
Therefore,
\begin{equation*}
\frac{1}{p}\langle u^p(\tau)\rangle + \biggl( \frac{4(p-1)}{p^2}- \frac{2}{p}\nu -  \frac{\beta}{2p\nu} \biggr) \int_0^\tau \langle (\nabla u^{\frac{p}{2}})^2 \rangle dt \leqslant \frac{1}{p}\langle f^p \rangle + \frac{\beta}{2p\nu}\int_0^\tau g(t) \langle u^p\rangle dt
\end{equation*}
The maximum of $\nu \mapsto \frac{4(p-1)}{p^2}- \frac{2}{p}\nu -  \frac{\beta}{2p\nu} $, attained at $\sqrt{\beta/4}$, is positive if and only if $p>(1-\sqrt{\beta/4})^{-1}$. Set $\nu:=\sqrt{\beta/4}$.
Estimating $\int_0^\tau g(t) \langle u^p\rangle dt \leqslant G(h) \sup_{t \in [0,\tau]}\langle u^p(t)\rangle$, and selecting $h$ sufficiently small, so that $1-\frac{\beta}{2\nu}G(h)>0$ (recall that $G(h)=o(h)$),
we obtain
\begin{equation*}
\frac{1}{p}\left(1-\frac{\beta}{2\nu}G(h) \right)\sup_{t \in [0,\tau]}\langle u^p(t)\rangle + \biggl( \frac{4(p-1)}{p^2}- \frac{2}{p}\nu -  \frac{\beta}{2p\nu} \biggr) \int_0^\tau \langle (\nabla u^{\frac{p}{2}})^2 \rangle dt \leqslant \frac{1}{p}\langle f^p \rangle.
\end{equation*}
which yields \eqref{est_Lp}.
\end{proof}

\end{document}